\documentclass[11pt, reqno]{amsart}
\usepackage[margin=1.2in]{geometry}          
\usepackage{graphicx}
\usepackage{amssymb}
\usepackage{epstopdf}
\usepackage{tikz-cd} 
\usepackage[colorlinks=true, pdfstartview=FitV, linkcolor=blue, citecolor=blue, urlcolor=blue]{hyperref}
\usepackage{tikz}
\usepackage[listings]{tcolorbox}
\usepackage{caption}      
\usepackage{subcaption}
\usepackage{mathrsfs}  
\usepackage{enumitem} 
\usepackage{verbatim} 
\usepackage[smalltableaux]{ytableau} 
\usepackage[nameinlink]{cleveref}
\usepackage{soul}
\usepackage{color}
\usepackage{xcolor}
\definecolor{darkgreen}{rgb}{0.0,0.7,0.0}
\definecolor{linkblue}{HTML}{1a0dab}

\definecolor{Peach}{HTML}{FFC78F}
\definecolor{PeachBorder}{HTML}{FF9900}

\definecolor{LightBlue}{HTML}{C7F5FF}
\definecolor{BlueBorder}{HTML}{0099CC}

\definecolor{LightPurple}{HTML}{E6E6FA}  
\definecolor{PurpleBorder}{HTML}{9370DB} 

\definecolor{LightYellow}{HTML}{FFFFE0}  
\definecolor{YellowBorder}{HTML}{FFD700} 

\definecolor{LightGreen}{HTML}{E6FFE6}   
\definecolor{GreenBorder}{HTML}{66CC66}  

\definecolor{LightLavender}{HTML}{F3E8FD}
\definecolor{LavenderBorder}{HTML}{B57EDC}

\definecolor{LightYellow2}{HTML}{FFF9C4}
\definecolor{YellowBorder2}{HTML}{FBC02D}

\definecolor{LightGreen2}{HTML}{E0F2F1}
\definecolor{GreenBorder2}{HTML}{26A69A}

\definecolor{LightPurple}{HTML}{EDE7F6}
\definecolor{PurpleBorder}{HTML}{7E57C2}

\definecolor{LightPink}{HTML}{F8BBD0}
\definecolor{PinkBorder}{HTML}{EC407A}

\definecolor{LightCyan}{HTML}{E0F7FA}
\definecolor{CyanBorder}{HTML}{00ACC1}

\newcommand{\kk}{{\sf k}} 

\newcommand{\Cl}{{\mathrm{Cl}}}

\newtheoremstyle{notheorem}
  {10pt}    
  {10pt}    
  {\itshape} 
  {}        
  {}        
  {}        
  {0pt}     
  {}        

\theoremstyle{notheorem}

\theoremstyle{plain} 
\newtheorem{thm}{Theorem}[section]

\newtheorem{introthm}{Theorem}

\newtheorem*{introthm*}{Theorem}

\newtheorem{cor}[thm]{Corollary}
\newtheorem{lem}[thm]{Lemma}
\newtheorem{prop}[thm]{Proposition}

\newtheorem*{oprobl*}{Open Problem}

\theoremstyle{definition}
\newtheorem{defn}[thm]{Definition}

\newtheorem{assumption}[thm]{Assumption}

\theoremstyle{remark}
\newtheorem{rem}[equation]{Remark}

\numberwithin{equation}{section}  
\newcounter{proofeq}
\renewcommand{\theproofeq}{\arabic{proofeq}}

\title{Principal symmetric ideals in the coordinate rings of curves}

\author[Vinuge Rupasinghe]{Vinuge Rupasinghe}
\address{}
\email{vinuge.dinusith@outlook.com}


\begin{document}
\maketitle

\begin{abstract}
The study of principal symmetric ideals (PSIs) in ambient polynomial rings was complicated by the combinatorial instability of minimal generators for ideal powers. We resolve this instability in the two variable case by translating the problem into the arithmetic geometry of symmetric affine plane curves. By working topdown within the Dedekind domain of a symmetric coordinate ring, we establish a precise geometric dictionary for PSIs. We prove that the prime factorization of a PSI is strictly determined by the $S_2$-orbits of its symmetric intersection locus, and that ramification corresponds exactly to tangential intersections, which are detected globally by a novel Symmetric Discriminant ideal. Crucially, we demonstrate that the ideal class of any PSI is a $2$-torsion element in the Ideal Class Group. This establishes that the powers of a PSI exhibit strict periodicity, alternating between being principal and requiring exactly two generators. Finally, we localize this arithmetic obstruction to the axis of symmetry, culminating in a Parity Criterion that determines principality based on intersection multiplicities along the diagonal.
\end{abstract}

\setcounter{tocdepth}{1} 
\tableofcontents

\section{Introduction}
\label{sec:intro}

The study of homogeneous principal symmetric ideals (PSIs) occupies a rich intersection between commutative algebra and invariant theory. In the standard ambient setting, one considers the polynomial ring $S = \kk[x_1, \ldots, x_d]$ equipped with the natural action of the symmetric group $S_d$ permuting the variables. Given a polynomial $f \in S$, the principal symmetric ideal $(f)_{S_d}$ is generated by the finite orbit of $f$ under this action. 

While structurally elegant, the arithmetic behavior of PSIs in the ambient space is notoriously complex. A central difficulty, observed in recent literature \cite{HSS, PSI, WAL25}, is the combinatorial explosion of minimal generators required for the power of a principal symmetric ideal, $I^n$, as the dimension $d$ and the power $n$ increase. In this paper, we refer to this unbounded growth as \emph{generator instability}. Because the ambient ring $S$ has dimension $d \ge 2$, it lacks the rigid, one dimensional arithmetic structure required to systematically resolve questions of ideal factorization and divisibility.

In this paper, we resolve this combinatorial instability by shifting the problem from the ambient polynomial ring into the realm of arithmetic geometry. We restrict our attention to the case $d=2$ and specialize to the coordinate ring $R = \kk[x,y]/(g)$ of a non-singular, symmetric affine plane curve. 

This geometric translation lowers the Krull dimension to one, ensuring that $R$ is a Dedekind domain. In this setting, the combinatorial problem of finding minimal generators is elegantly reinterpreted as an arithmetic obstruction measured by the Ideal Class Group, $\Cl(R)$. By studying the action of the Galois involution $\sigma(x,y) = (y,x)$ on the fractional ideals of $R$, we map the algebraic properties of PSIs directly onto the local and global geometry of the underlying curve.

\subsection{Statement of Main Results}
Our investigation yields a complete geometric dictionary for the arithmetic of Principal Symmetric Ideals. We demonstrate that the algebraic factorization of $I_f = (f, \sigma f)R$ is strictly determined by the $S_2$-orbits of the intersection between the curve and the symmetric locus defined by $f$. 

Furthermore, we show that the failure of square-free factorization ramification is not an abstract algebraic anomaly, but a highly visual geometric phenomenon governed by tangency. 

\begin{introthm}[Geometric Factorization and Ramification, \Cref{thm:factorization} \& \Cref{thm:tangency}]
The prime factorization of a Principal Symmetric Ideal $I_f \subset R$ is uniquely determined by the orbits of the Symmetric Intersection Locus $Z(f)$. Furthermore, the ideal is ramified at a prime factor $\mathfrak{p}$ if and only if the intersection of the base curve $V(g)$ and the generator locus $V(f)$ is tangential at the corresponding point. Globally, this ramification is detected by the vanishing of the Symmetric Discriminant ideal, $\Delta(f)$.
\end{introthm}

Having established the local geometry, we address the global arithmetic. We introduce a Norm Principle for symmetric coordinate rings, proving that the product of any ideal with its Galois conjugate yields a principal ideal generated by a symmetric element. Applying this to PSIs, we resolve the generator instability problem by proving that the obstruction to a PSI being principal is strictly bounded by the torsion of the Class Group.

\begin{introthm}[PSI Torsion and Periodicity, \Cref{thm:torsion}]
Let $I_f$ be a Principal Symmetric Ideal in $R$. Then the ideal class $[I_f]$ is a $2$-torsion element in the Ideal Class Group $\Cl(R)$. Consequently, the powers of $I_f$ exhibit strict periodicity: $I_f^n$ is a principal ideal for all even $n$, and requires exactly two generators for all odd $n$.
\end{introthm}

Finally, we localize this global Class Group obstruction to the axis of symmetry. By conceptually "folding" the curve along the diagonal $y=x$, we show that off-diagonal intersection points act as conjugate pairs that perfectly cancel each other out in the Class Group. The entire arithmetic obstruction is therefore localized to the diagonal points. We formalize this via the \emph{Symmetric Index}, $\operatorname{Ind}_{sym}(f)$. In the hyperelliptic setting, this yields a highly efficient numerical test for principality.

\begin{introthm}[The Parity Criterion, \Cref{cor:parity}]
Let $R$ be the coordinate ring of a hyperelliptic curve. A Principal Symmetric Ideal $I_f$ is a principal ideal if and only if the generator $f$ intersects the diagonal $y=x$ with even total multiplicity.
\end{introthm}

\subsection{Organization of the Paper}
The paper is structured as follows. In \Cref{sec:psi}, we establish the foundational architecture, transitioning from the ambient polynomial ring to the Dedekind domain of a symmetric coordinate ring. In \Cref{chap: Geometric Factorization}, we formalize the Geometric Factorization theorem and classify the prime orbits. \Cref{chap: Ramification} introduces the machinery of Kähler differentials to build the Symmetric Discriminant ideal, providing a global algebraic test for ramification via tangency. Finally, in \Cref{chap: Class Group}, we introduce the Galois action on the Class Group, prove the $2$-torsion stability of PSIs, and define the Symmetric Index to establish the Parity Criterion.

\section{Preliminaries: Ambient Space to Curve Geometry}
\label{sec:psi}

\subsection{The Ambient Setting}
In this section, we formally define the central object of our study: the Principal Symmetric Ideal (PSI). We begin by reviewing the definition in the ambient polynomial ring, as used in previous literature \cite{HSS}, before specializing to the coordinate ring of a symmetric curve where our arithmetic results take place.

\begin{defn}[Principal Symmetric Ideal]
Let $\kk$ be a field and let $S = \kk[x_1,\ldots,x_d]$ be the polynomial ring in $d$ variables. 
The symmetric group $S_d$ acts on $S$ by permuting the variables:
\[
\sigma \cdot f(x_1,\ldots,x_d) = f(x_{\sigma(1)},\ldots,x_{\sigma(d)}),
\quad \sigma \in S_d.
\]
Given a polynomial $f \in S$, the \emph{principal symmetric ideal} generated by $f$, abbreviated as a \emph{PSI}, is the ideal generated by the orbit of $f$ under this action. Explicitly,
\[
(f)_{S_d} = \bigl( f(x_{\sigma(1)}, \ldots, x_{\sigma(d)}) \mid \sigma \in S_d \bigr) \subseteq S.
\]
\end{defn}

In this ambient setting, the study of PSIs is primarily combinatorial. As noted in \cite{WAL25}, questions regarding the number of generators for powers of PSIs become increasingly complex as the dimension increases, often lacking a stable arithmetic behavior. Furthermore, the ambient ring $S$ has dimension $d \ge 2$, which precludes the use of powerful one dimensional arithmetic tools such as Dedekind domains and divisor theory. To resolve this structural bottleneck, we shift our focus from the ambient space to the intrinsic geometry of a curve. By restricting our domain to the coordinate ring of a symmetric plane curve, we drop the dimension to one. In this setting, the combinatorial instability of PSI generators can be elegantly reinterpreted through the lens of ideal class groups and arithmetic geometry.

\subsection{The Symmetric Coordinate Ring}
\begin{defn}[Symmetric Coordinate Ring]
Let $\kk$ be a field and let $g \in \kk[x,y]$ be an irreducible polynomial such that the ideal
$(g)$ is invariant under the natural action of the symmetric group $S_2$ on $\kk[x,y]$.  
The \emph{symmetric coordinate ring} of the affine plane curve defined by $g(x,y)=0$ is the
quotient ring $R := \kk[x,y]/(g)$, equipped with the induced $S_2$-action.
\end{defn}

\begin{defn}[Dedekind PSI]
Let $f \in R$ be a non-zero element. The \textbf{Principal Symmetric Ideal} in $R$ is the ideal generated by the orbit of $f$:
\[ I_f = \big( f, \, \sigma(f) \big)R. \]
\end{defn}

\begin{prop}[Descent of Symmetry]
The action of $S_2$ on $S$ induces a well-defined automorphism on the quotient ring $R = S/(g)$ if and only if the ideal $(g)$ is stable under $S_2$. In particular, if $g$ is a symmetric polynomial, the action is well-defined.
\end{prop}

\begin{proof}
This is a standard algebraic result regarding group actions descending to quotient rings by invariant ideals.
\end{proof}

Having established that the $S_2$-action descends naturally, we must understand the internal algebraic structure of $R$ relative to its symmetric invariants. 

\begin{defn}[Elementary Symmetric Polynomials]
The elementary symmetric polynomials in the variables $x$ and $y$ are defined as $e_1 = x + y$ and $e_2 = xy$. These polynomials generate the invariant subring $\kk[x,y]^{S_2}$.
\end{defn}

\begin{prop}[Quadratic Representation]\label[prop]{prop:quadratic rep}
The coordinate ring $R$ functions as a quadratic extension of the invariant subring $R^{S_2}$. Specifically, every element $h \in R$ can be expressed uniquely in the form $h = a + b \cdot x$, where $a, b \in R^{S_2}$ are symmetric polynomials.
\end{prop}

\begin{proof}
This is a standard consequence of the fundamental theorem of symmetric polynomials; $x$ and $y$ are roots of the generic quadratic $T^2 - e_1 T + e_2 = 0$ over the invariant subring.
\end{proof}

\begin{rem}[Generalization to Higher Dimensions]
While restricting to one-dimensional symmetric coordinate rings allows us to utilize the rigid arithmetic of Dedekind domains (where ideal factorization strictly equates to prime powers), this geometric philosophy naturally extends to higher dimensions. In a higher dimensional symmetric quotient ring, the analogue of our geometric prime factorization would be the study of the primary decomposition of Principal Symmetric Ideals. We leave the exploration of this higher dimensional primary decomposition to future work.
\end{rem}

\subsection{Arithmetic Regularity and Factorization}
\begin{assumption}
Throughout this paper, we assume that $g \in \kk[x,y]$ is an irreducible, symmetric polynomial where $V(g)$ defines a non-singular curve.
\end{assumption}

With the algebraic basis established, we turn to the arithmetic properties of $R$. In a general polynomial ring $\kk[x_1, \dots, x_n]$, unique factorization of ideals inevitably fails. However, our geometric assumption of non-singularity ensures that $R$ exhibits the rigid arithmetic structure necessary to recover unique factorization.

\begin{lem}\label[lem]{lem:dedekind}
Let $R = \kk[x,y]/(g)$, where $g \in \kk[x,y]$ defines a non-singular affine plane curve. Then $R$ is a Dedekind domain. Consequently, every proper, non-zero ideal $I \subset R$ admits a unique factorization into prime ideals.
\end{lem}
\begin{proof}
This is a standard result in the arithmetic of algebraic curves, for instance see \cite[Chapter VII, Corollary 2.7, p. 229]{LOR}. The consequent unique factorization of ideals is detailed in \cite[Chapter III, Theorem 2.8, p. 91]{LOR}.
\end{proof}

By the unique factorization of ideals, any proper non zero ideal $I$ can be expressed as $I = \mathfrak{p}_1^{e_1} \cdots \mathfrak{p}_k^{e_k}$. The exponents $e_i$ in this factorization can be naturally interpreted via discrete valuations $v_{\mathfrak{p}_i}$ corresponding to the localizations $R_{\mathfrak{p}_i}$. While the fundamental theorem of ideal arithmetic guarantees that every symmetric ideal factors uniquely into primes, it does not guarantee that these ideals are principal. A central question of this paper is determining precisely when a symmetric ideal is principal an obstruction measured entirely by the Ideal Class Group.

\section{Geometric Factorization}
\label{chap: Geometric Factorization}

In this section, we present the first major result of this paper: a geometric characterization of the prime factorization of PSIs. By transitioning to the coordinate ring of a symmetric curve, we demonstrate that the algebraic problem of factoring $I_f = (f, \sigma f)R$ is strictly equivalent to the geometric problem of computing the intersection of the curve with the symmetric locus defined by $f$. 

A key advantage of our approach over ambient ring methods \cite{WAL25} is the rigid bound on minimal generators. Because our symmetric coordinate ring $R$ is a Dedekind domain, every ideal is generated by at most two elements. Thus, the symmetric ideal $I_f$ is intrinsically bounded, generated exclusively by the set $\{f, \sigma f\}$ (or solely by $f$ if the ideal is already principal).

\subsection{The Symmetric Intersection Locus}
To compute the factorization of $I_f$, we must locate the prime ideals containing it. Under the dictionary of algebraic geometry, these prime ideals correspond directly to specific points on the underlying curve.

\begin{defn}[Symmetric Intersection Locus]
Let $C(\bar{\kk})$ denote the set of points of the curve $C$ over the algebraic closure $\bar{\kk}$. The \emph{Symmetric Intersection Locus} of $f$, denoted $Z(f)$, is the set of points on the curve where both the generator and its conjugate vanish simultaneously 
\[ Z(f) = \{ P \in C(\bar{\kk}) \mid f(P) = 0 \text{ and } f(\sigma P) = 0 \}. \]
\end{defn}

\begin{lem}[Symmetry of the Locus]
The set $Z(f)$ is invariant under the action of $S_2$. In particular, if $P \in Z(f)$, then $\sigma(P) \in Z(f)$.
\end{lem}

\begin{proof}
This follows immediately from the definition of $Z(f)$ and the fact that $\sigma$ is an involution ($\sigma^2 = \operatorname{id}$).
\end{proof}

\begin{thm}[Geometric Factorization]\label{thm:factorization}
Let $I_f \subset R$ be a Principal Symmetric Ideal. The prime factorization of $I_f$ is uniquely determined by the orbits of the intersection locus $Z(f)$. Specifically:
\[ I_f = \prod_{\mathcal{O} \subseteq Z(f)} \mathfrak{P}_{\mathcal{O}}^{v_{\mathcal{O}}} \]
where $\mathfrak{P}_{\mathcal{O}}$ is the ideal associated to the orbit $\mathcal{O}$, and the exponent $v_{\mathcal{O}}$ is the local valuation of the ideal at any point $P \in \mathcal{O}$.
\end{thm}
\begin{proof}
    Because $R$ is a Dedekind domain by \Cref{lem:dedekind}, the non-zero proper ideal $I_f$ admits a unique factorization into prime ideals. By the point maximal ideal correspondence, every non-zero prime ideal in $R$ is of the form $\mathfrak{p}_P$ for a unique point $P \in C(\bar{\kk})$. Thus, we may express the factorization of $I_f$ as the formal product:
    $$ I_f = \prod_{P \in C} \mathfrak{p}_P^{v_P(I_f)} $$
    where $v_P(I_f)$ is the discrete valuation of $I_f$ at the localization $R_{\mathfrak{p}_P}$, and $v_P(I_f) \ge 0$ for all $P$, with only finitely many exponents strictly greater than zero.

    First, we determine the support of this ideal. A prime ideal $\mathfrak{p}_P$ appears in the factorization of $I_f$ with a strictly positive exponent if and only if $I_f \subseteq \mathfrak{p}_P$. Because $I_f$ is generated by $\{f, \sigma f\}$, this containment holds if and only if both $f \in \mathfrak{p}_P$ and $\sigma f \in \mathfrak{p}_P$. 
    
    Evaluating these elements at the point $P$, the condition $f \in \mathfrak{p}_P$ is equivalent to $f(P) = 0$. For the conjugate generator, observe that the action of $\sigma$ on the coordinate ring corresponds geometrically to the coordinate involution, meaning $(\sigma f)(P) = f(\sigma P)$. Therefore, $\sigma f \in \mathfrak{p}_P$ is equivalent to $f(\sigma P) = 0$. Both conditions are satisfied simultaneously if and only if $P \in Z(f)$. Consequently, $v_P(I_f) > 0$ strictly when $P \in Z(f)$, allowing us to restrict the product to the symmetric intersection locus:
    $$ I_f = \prod_{P \in Z(f)} \mathfrak{p}_P^{v_P(I_f)} $$

    Next, we organize this product by the orbits of the $S_2$-action. The symmetric group $S_2$ acts on the set $Z(f)$, partitioning it into disjoint orbits $\mathcal{O}$. For any point $P \in Z(f)$, the corresponding orbit is either a single point $\mathcal{O} = \{P\}$ (if $P$ lies on the diagonal $y=x$) or a pair of distinct points $\mathcal{O} = \{P, \sigma P\}$. 
    
    We claim that the valuation is constant on orbits; that is, $v_P(I_f) = v_{\sigma P}(I_f)$. To see this, apply the automorphism $\sigma$ to the ideal $I_f$. By definition, $I_f = (f, \sigma f)$, so $\sigma(I_f) = (\sigma f, \sigma^2 f) = (\sigma f, f) = I_f$. The ideal is strictly invariant under $\sigma$. Applying $\sigma$ to the prime factorization yields:
    $$ I_f = \sigma(I_f) = \sigma\left( \prod_{P \in Z(f)} \mathfrak{p}_P^{v_P(I_f)} \right) = \prod_{P \in Z(f)} \sigma(\mathfrak{p}_P)^{v_P(I_f)} = \prod_{P \in Z(f)} \mathfrak{p}_{\sigma P}^{v_P(I_f)} $$
    Because the prime factorization in a Dedekind domain is unique, the exponent of $\mathfrak{p}_{\sigma P}$ in this expanded product must equal the exponent of $\mathfrak{p}_{\sigma P}$ in the original factorization. Therefore, $v_P(I_f) = v_{\sigma P}(I_f)$. 

    Let $v_{\mathcal{O}}$ denote this common uniform valuation for all points in the orbit $\mathcal{O}$. By defining $\mathfrak{P}_{\mathcal{O}} = \prod_{P \in \mathcal{O}} \mathfrak{p}_P$, we can group the prime factors orbit by orbit:
    $$ I_f = \prod_{\mathcal{O} \subseteq Z(f)} \left( \prod_{P \in \mathcal{O}} \mathfrak{p}_P \right)^{v_{\mathcal{O}}} = \prod_{\mathcal{O} \subseteq Z(f)} \mathfrak{P}_{\mathcal{O}}^{v_{\mathcal{O}}} $$
    This concludes the proof.
\end{proof}

\subsection{Visualizing the Locus}
At first glance, \Cref{thm:factorization} may appear abstract. To clarify the mechanics of the theorem, it is instructive to adopt a strictly geometric perspective. In a Dedekind domain, factoring an ideal is geometrically synonymous with identifying points on a curve alongside their intersection multiplicities.

To visualize the ideal $I_f = (f, \sigma f)$ within $R = \kk[x,y]/(g)$, we treat the affine plane curve $C = V(g)$ as the ambient space. The generators impose two simultaneous constraints: the vanishing of $f$, defining the curve $Z_f(\kk)$, and the vanishing of its conjugate $\sigma f$, defining $Z_{\sigma f}(\kk)$. Crucially, $Z_{\sigma f}(\kk)$ is the geometric reflection of $Z_f(\kk)$ across the diagonal line $y=x$.

Consequently, a point $P$ lies in the support of the ideal $I_f$ if and only if it satisfies all three conditions simultaneously:
\[ P \in Z(f) \quad \Longleftrightarrow \quad P \in C \cap Z_f(\kk) \cap Z_{\sigma f}(\kk). \]
Since $C$ is invariant under the involution $\sigma : (x,y) \mapsto (y,x)$, and the generating set is symmetrically closed, the intersection locus $Z(f)$ is inherently $S_2$-invariant. Rather than appearing as isolated points, the prime ideals naturally organize into orbits under the $S_2$-action, appearing in symmetric "packets."

\subsection{Classification of Prime Orbits}
Having established that the prime factors are governed by geometric orbits, we can classify these factors based entirely on their geometric position relative to the axis of symmetry, $y=x$.

\begin{thm}[Classification of Prime Factors of $I_f$]
Let $R = \kk[x,y]/(g)$ be a symmetric coordinate ring, and let $I_f = (f,\sigma f) \subset R$. Then the prime ideals appearing in the factorization of $I_f$ are classified according to the geometry of the $S_2$-orbits in $Z(f)$. Exactly one of the following occurs for each orbit $\mathcal{O} \subset Z(f)$:
\begin{enumerate}
\item \emph{(Symmetric primes - diagonal case).} If $\mathcal{O} = \{P\}$ consists of a single point, then $P$ lies on the diagonal $y=x$ and is fixed by the involution $\sigma$. The corresponding prime ideal $\mathfrak{p}$ satisfies $\sigma(\mathfrak{p}) = \mathfrak{p}$.
\item \emph{(Conjugate primes - split case).} If $\mathcal{O} = \{P,\sigma P\}$ consists of two distinct points, then $P$ does not lie on the diagonal. The corresponding prime ideals occur in conjugate pairs $\mathfrak{p}$ and $\sigma(\mathfrak{p})$, and the ideal $I_f$ contains the symmetric product $\mathfrak{p}\cdot\sigma(\mathfrak{p})$.
\end{enumerate}
\end{thm}
\begin{proof}
    By \Cref{thm:factorization}, the prime factors of $I_f$ correspond to the points in the symmetric intersection locus $Z(f)$, and these factors are grouped by the orbits of the $S_2$-action. Let $\mathcal{O} \subseteq Z(f)$ be such an orbit. Since $S_2$ is a group of order $2$ acting via the involution $\sigma$, the orbit-stabilizer theorem dictates that the size of the orbit $\mathcal{O}$ must be strictly either $1$ or $2$. We proceed by analyzing these two exhaustive cases.

    \textbf{Case 1: (Symmetric primes)} Suppose $|\mathcal{O}| = 1$. Then $\mathcal{O} = \{P\}$ for a single point $P \in Z(f)$. This implies that $P$ is a fixed point of the involution, meaning $\sigma(P) = P$. Let the coordinates of $P$ in $\mathbb{A}^2_k$ be $(a,b)$. By the definition of the geometric action, $\sigma(a,b) = (b,a)$. The fixed-point condition $(a,b) = (b,a)$ forces $a = b$, demonstrating that $P$ lies exactly on the diagonal line $y = x$. Let $\mathfrak{p}$ be the maximal ideal corresponding to $P$, which is of the form $\langle x-a, y-a \rangle / (g)$. Applying the algebraic involution $\sigma$ to the generators of this ideal yields:
    $$ \sigma(\mathfrak{p}) = \sigma \left( \frac{\langle x-a, y-a \rangle}{(g)} \right) = \frac{\langle y-a, x-a \rangle}{(g)} = \mathfrak{p}. $$
    Thus, the prime ideal itself is strictly invariant under the $S_2$-action.

    \textbf{Case 2: (Conjugate primes)} Suppose $|\mathcal{O}| = 2$. Then $\mathcal{O} = \{P, \sigma P\}$ consists of two geometrically distinct points. Let $P = (a,b)$. Because $P \neq \sigma P$, it follows that $(a,b) \neq (b,a)$, meaning $a \neq b$. Therefore, neither $P$ nor $\sigma P$ lies on the diagonal. Let $\mathfrak{p}$ be the prime ideal corresponding to $P$, generated by $\langle x-a, y-b \rangle / (g)$. Applying the involution $\sigma$ to this ideal gives:
    $$ \sigma(\mathfrak{p}) = \sigma \left( \frac{\langle x-a, y-b \rangle}{(g)} \right) = \frac{\langle y-a, x-b \rangle}{(g)} = \frac{\langle x-b, y-a \rangle}{(g)}, $$
    which is precisely the prime ideal corresponding to the distinct mirror point $\sigma P$. Thus, the prime ideals corresponding to the points in this orbit form a conjugate pair, $\mathfrak{p}$ and $\sigma(\mathfrak{p})$. 
    
    By \Cref{thm:factorization}, the ideals $\mathfrak{p}$ and $\sigma(\mathfrak{p})$ both appear in the prime factorization of $I_f$ with the exact same exponent $v_{\mathcal{O}} \ge 1$. Since $P \neq \sigma P$, the corresponding maximal ideals $\mathfrak{p}$ and $\sigma(\mathfrak{p})$ are distinct and therefore comaximal. Consequently, their product divides $I_f$, which in the language of ideals means $I_f \subseteq \mathfrak{p} \cdot \sigma(\mathfrak{p})$. This establishes that $I_f$ is contained within the symmetric product of the conjugate pair.
\end{proof}

\section{Ramification and Geometric Singularity}
\label{chap: Ramification}

In \Cref{chap: Geometric Factorization}, we established the geometric factorization theorem, demonstrating that the prime factors of PSIs are in natural correspondence with the intersection points of their associated curves. In the present chapter, we refine this analysis by determining the \emph{multiplicities} with which these prime factors occur. Our goal is to characterize precisely when a symmetric ideal fails to be square-free and to relate this algebraic phenomenon to the local geometry of the curve.

Although Dedekind domains admit unique factorization of ideals into prime powers, this does not inherently guarantee that the resulting ideals are radical. The appearance of higher powers of prime ideals reflects additional geometric structure invisible at the level of set theoretic intersections, described as \emph{ramification}. We show that the exponent $e_i$ in the factorization $I = \prod_i \mathfrak{p}_i^{e_i}$ measures the order of contact of the corresponding components. 

\subsection{The Local Tangency Criterion}

\begin{thm}[Tangency Criterion]\label{thm:tangency}
Let $I_f$ be a Principal Symmetric Ideal. Let $P \in Z(f)$ be an intersection point corresponding to the prime ideal $\mathfrak{p}$. The following are equivalent:
\begin{enumerate}
    \item The ideal is ramified at $\mathfrak{p}$ (i.e., $e_{\mathfrak{p}} > 1$).
    \item The intersection of the curve $C$ and the generator locus $V(f)$ at $P$ is tangential.
\end{enumerate}
\end{thm}
\begin{proof}
    We study the intersection locally by passing to the local ring of the curve $C$ at the point $P$, denoted $\mathcal{O}_{C,P}$. Because we assumed $C$ is a non-singular curve, $\mathcal{O}_{C,P}$ is a discrete valuation ring with a unique maximal ideal $\mathfrak{m}_P$ corresponding to $\mathfrak{p}$, and an associated discrete valuation $v_P$. By the properties of Dedekind domains, the prime ideal $\mathfrak{p}$ appears in the unique factorization of $I_f$ with exponent $e_{\mathfrak{p}} = v_P(I_f)$. Thus, the ideal is ramified at $\mathfrak{p}$ if and only if $e_{\mathfrak{p}} \ge 2$, which is algebraically equivalent to the local containment condition $I_f \mathcal{O}_{C,P} \subseteq \mathfrak{m}_P^2$. Since $I_f$ is generated by the set $\{f, \sigma f\}$, this containment is satisfied if and only if its generators individually reside in $\mathfrak{m}_P^2$, focusing on the primary generator, we require $f \in \mathfrak{m}_P^2$, which means $v_P(f) \ge 2$. We translate this local algebraic containment into geometry by examining the Zariski cotangent space $\mathfrak{m}_P/\mathfrak{m}_P^2$. A valuation of $v_P(f) = 1$ indicates that $f \in \mathfrak{m}_P \setminus \mathfrak{m}_P^2$, meaning the image of $f$ forms a basis for the one-dimensional cotangent space and acts as a local uniformizer. Geometrically, because this linear term does not vanish, the Zariski tangent spaces $T_P(C)$ and $T_P(V(f))$ intersect transversely. Conversely, the ramification condition $v_P(f) \ge 2$ forces $f \in \mathfrak{m}_P^2$, meaning the class of $f$ in the cotangent space is strictly zero. Geometrically, the vanishing of this linear component implies that the differential defining $T_P(V(f))$ restricts to zero everywhere along $T_P(C)$, forcing the linear dependence $T_P(C) \subseteq T_P(V(f))$, which is the precise definition of a tangential intersection. Therefore, the algebraic ramification condition is satisfied if and only if the geometric tangency condition holds, and by the $S_2$-symmetry of the intersection locus, the identical local behavior governs the conjugate $\sigma f$, completing the equivalence.
\end{proof}
To bridge the gap between geometric intuition and algebraic structure, recall that every point $P$ on an algebraic curve corresponds to a unique prime ideal $\mathfrak{p}_P$ in its coordinate ring. This correspondence translates an \textit{intersection} directly into \textit{ideal membership}: a function $f$ has a root at $P$ if and only if $f \in \mathfrak{p}_P$. However, simple membership fails to distinguish between a curve that cuts transversely through an axis and one that merely grazes it. To capture this nuance, we must examine the higher powers of the maximal ideal.

Consider the local behavior of $f$ near $P$. In the generic case of a \textit{transversal intersection}, the function behaves linearly, falling into $\mathfrak{p}_P$ but avoiding $\mathfrak{p}_P^2$, yielding a simple valuation of 1. In contrast, a \textit{tangential intersection} occurs when the curves share a tangent line. Here, the linear term vanishes, forcing $f$ into the square of the prime ideal, $\mathfrak{p}_P^2$. The valuation jumps to at least 2, yielding a ramified prime factor $\mathfrak{p}^2$.

\subsection{A Coordinate-Free Approach via Kähler Differentials}
Applying the local tangency criterion point by point is computationally intractable. We require a global method to detect these grazing intersections. To achieve this, we apply the machinery of Kähler differentials and exterior algebra.

Let $R = \kk[x,y]$. The geometry of the intersection between the plane curves defined by $f \in R$ and $g \in R$ is natively encoded in the linear dependence of their differentials $df$ and $dg$ within the module $\Omega_{R/\kk}$.

\begin{thm}[Global Tangency Criterion]\label{thm:jacobian}
The curves defined by $V(f)$ and $V(g)$ intersect tangentially at a point $p$ if and only if their differentials are linearly dependent in the fiber $\Omega_{R/\kk} \otimes \kk(p)$. Globally, this dependence is characterized by the vanishing of the Jacobian determinant $\det(J(g,f))$.
\end{thm}
\begin{proof}
This is an immediate consequence of the standard Jacobian criterion for transversality and the properties of Kähler differentials. See \cite[Section 3.5, Theorem 3.18, p. 196]{SHA}.
\end{proof}

\begin{rem}
This establishes that the Jacobian determinant is not merely an analytic artifact, but the coordinate representation of the intrinsic area form $df \wedge dg$.
\end{rem}

\subsection{The Symmetric Discriminant}
In the language of calculus, the vanishing of the Jacobian determinant $\det(J) = 0$ corresponds to the gradient vectors being parallel. By combining this determinant with our original defining equations, we can package the entire geometric tangency condition into a single, global algebraic object.

\begin{defn}[Symmetric Discriminant Ideal]
Let $g,f \in \kk[x,y]$. The \emph{Symmetric Discriminant ideal} associated to $f$ (with respect to $g$) is
\[ \Delta(f) := (\, g,\ f,\ \det(J(g,f)) \,) \subset \kk[x,y]. \]
\end{defn}

\begin{defn}[Tangency Locus]
The \emph{tangency locus} of $g$ and $f$ is the algebraic subset $\mathcal{T}(g,f) := V\!\bigl(\Delta(f)\bigr)$. Equivalently, it consists of points on $V(g)$ at which the hypersurfaces defined by $g$ and $f$ fail to meet transversely.
\end{defn}

This construction yields a powerful computational corollary, transforming the abstract problem of detecting ramification into a concrete algebraic test over the ambient ring.

\begin{cor}
Let $f,g \in \kk[x,y]$, and let $I_f$ denote the principal symmetric ideal generated by $f$. Then $I_f$ is \emph{ramified} if and only if the Symmetric Discriminant ideal $\Delta(f)$ is not the unit ideal.
\end{cor}
\begin{proof}
If $1 \in \Delta(f)$, the Weak Nullstellensatz implies that the system of equations defined by its generators has no common solution anywhere in the algebraic closure $\bar{\kk}$. Consequently, there is no point on the curve where $f$ vanishes with parallel gradients, meaning no point of tangency exists; therefore, the ideal is unramified. Conversely, if $\Delta(f)$ is a proper ideal, the Weak Nullstellensatz guarantees the existence of a common root $P \in \mathbb{A}^2_{\bar{\kk}}$. Because this root must simultaneously satisfy $g(P)=0$, $f(P)=0$, and $\det(J(g,f))(P) = 0$, the intersection at $P$ is strictly tangential. This guarantees that the ideal is ramified.
\end{proof}

\section{The Class Group Obstruction}
\label{chap: Class Group}

We now shift our perspective from the local to the global, addressing the arithmetic structure of these ideals within the coordinate ring as a whole. A central difficulty in the study of principal symmetric ideals is the \textbf{instability} of their generators \cite{WAL25}. In the ambient polynomial ring, the number of elements required to generate the power $I^n$ often grows combinatorially. By restricting our attention to the symmetric coordinate ring $R$, we resolve the generator instability problem by mapping PSIs into the Ideal Class Group, $\Cl(R)$.

\subsection{Galois Action and The Norm Principle}
Let $\sigma: R \to R$ be the automorphism induced by permuting variables. For any fractional ideal $J \subset R$, its conjugate ideal is defined as $\sigma(J) = \{ \sigma(r) \mid r \in J \}$. Because $\sigma$ respects principal ideals, this action descends cleanly to the quotient group.

\begin{defn}[Galois Action on the Class Group]
Let $R$ be a symmetric coordinate ring. The action of $S_2 = \{1, \sigma\}$ induces a well-defined automorphism on $\Cl(R)$. For any ideal class $\mathcal{C} \in \Cl(R)$ represented by a fractional ideal $J$, we define $\sigma \cdot [J] = [\sigma(J)]$.
\end{defn}

\Cref{prop:quadratic rep} established that $R$ behaves as a quadratic extension over its invariant subring $R^{S_2}$. Just as multiplying by a complex conjugate eliminates the imaginary component, multiplying an ideal $J$ by its Galois conjugate $\sigma(J)$ eliminates its \textbf{asymmetry}. The resulting product is generated by symmetric elements. Because the base ring of symmetric polynomials behaves like a Principal Ideal Domain (PID) in this context, the symmetrized ideal becomes principal.

\begin{lem}[The Norm Principle]\label{lem:norm}
Let $J$ be any fractional ideal in $R$. The product ideal $J \cdot \sigma(J)$ is a principal ideal generated by a symmetric element. Consequently, they act as inverses in the Class Group:
\[ [J] \cdot [\sigma(J)] = 1. \]
\end{lem}
\begin{proof}
    Let $K = \operatorname{Frac}(R)$ be the field of fractions of $R$, and let $K^{S_2}$ be its invariant subfield. The automorphism $\sigma$ extends to $K$, making $K / K^{S_2}$ a Galois extension of degree 2. We define the algebraic norm $N: K \to K^{S_2}$ by $N(x) = x \sigma(x)$.

    Let $J \subset K$ be a fractional ideal of $R$. We first show that the product ideal $J \cdot \sigma(J)$ is generated entirely by elements of the invariant subring $R^{S_2}$. Define $\mathfrak{a}$ to be the $R^{S_2}$-submodule generated by the norms of elements in $J$:
    \[ \mathfrak{a} = \bigl( N(x) \mid x \in J \bigr)_{R^{S_2}}. \]
    Clearly, $\mathfrak{a} R \subseteq J \cdot \sigma(J)$. To see the reverse inclusion, let $x, y \in J$. The mixed generators of $J \cdot \sigma(J)$ are of the form $x \sigma(y)$. Observe that the trace form $\operatorname{Tr}(x \sigma(y)) = x \sigma(y) + y \sigma(x)$ can be written as:
    \[ x \sigma(y) + y \sigma(x) = N(x+y) - N(x) - N(y). \]
    Since $J$ is an ideal, $x+y \in J$, and thus this trace element lies in $\mathfrak{a}$. Because $\{1, x\}$ forms a basis for $R$ over $R^{S_2}$ (as established in \Cref{prop:quadratic rep}), these symmetric trace and norm elements generate the entire product ideal. Hence, $J \cdot \sigma(J) = \mathfrak{a} R$.

    Geometrically, the invariant subring $R^{S_2}$ corresponds to the affine coordinate ring of the quotient space $X / S_2$. We can explicitly identify this quotient by applying the coordinate transformation $u = x+y$ and $v = x-y$,  Observe that the symmetric involution $\sigma(x,y) = (y,x)$ acts on this new basis via $\sigma(u) = y+x = u$ and $\sigma(v) = y-x = -v$. This explicitly identifies $\sigma$ with the standard hyperelliptic involution $\tau(u,v) = (u,-v)$. Algebraically, any element in the coordinate ring can be uniquely expressed in the form $A(u) + B(u)v$. The invariant subring $R^{S_2}$ consists precisely of the elements fixed by $\tau$. The condition $\tau(A(u) + B(u)v) = A(u) - B(u)v$ forces $B(u) = 0$ (assuming $\operatorname{char}(\kk) \neq 2$). Therefore, the invariant subring is generated entirely by the single free variable $u$, yielding $R^{S_2} \cong \kk[u]$. Geometrically, this confirms that the invariant quotient is isomorphic to the affine line $\mathbb{A}^1$. Because $\kk[u]$ is a Principal Ideal Domain (PID), every fractional ideal in $R^{S_2}$ is principal. 
    
    Therefore, $\mathfrak{a} = (h)$ for some symmetric element $h \in K^{S_2}$. Extending this back to $R$, we obtain:
    \[ J \cdot \sigma(J) = hR, \]
    which is a principal ideal. 
    
    Passing to the Ideal Class Group $\Cl(R)$, the class of any principal ideal is the identity. Thus, $[J \cdot \sigma(J)] = [(h)] = 1$. Since the class mapping is a homomorphism, we conclude $[J] \cdot [\sigma(J)] = 1$.
\end{proof}

\subsection{Torsion and Periodicity of Powers}
We now apply the Norm Principle directly to the Principal Symmetric Ideal $I_f = (f, \sigma f)R$.

\begin{thm}[PSI Torsion]\label{thm:torsion}
Let $I_f$ be a Principal Symmetric Ideal in $R$. Then the ideal class $[I_f]$ is a $2$-torsion element in the Class Group. That is: $[I_f]^2 = 1.$
\end{thm}
\begin{proof}
By definition, the ideal $I_f$ is generated by $\{f, \sigma f\}$. Applying $\sigma$ yields $\{\sigma f, f\}$. Since the generating set is identical, the ideal itself is invariant: $\sigma(I_f) = I_f$. This implies its class is fixed under the Galois action: $[\sigma(I_f)] = [I_f]$. However, by the Norm Principle (\Cref{lem:norm}), $[\sigma(I_f)] = [I_f]^{-1}$. Equating these expressions yields $[I_f] = [I_f]^{-1}$, and multiplying both sides by $[I_f]$ gives $[I_f]^2 = 1$.
\end{proof}

\begin{cor}[Periodicity of Powers]
Let $I_f$ be a Principal Symmetric Ideal and let $n \ge 1$. If $n$ is even, then $I_f^n$ is a principal ideal. If $n$ is odd, then $[I_f^n] = [I_f]$; in particular, if $I_f$ is not principal, $I_f^n$ remains non-principal and requires exactly two generators.
\end{cor}
\begin{proof}
Since the mapping from the monoid of fractional ideals to the Ideal Class Group is a homomorphism, we have $[I_f^n] = [I_f]^n$ for any integer $n \ge 1$. We proceed by analyzing the parity of $n$.

\textbf{Case 1 ($n$ is even):} Let $n = 2k$ for some integer $k \ge 1$. By \Cref{thm:torsion}, $[I_f]$ is a $2$-torsion element, meaning $[I_f]^2 = 1$. Therefore, 
\[ [I_f^n] = [I_f]^{2k} = \left([I_f]^2\right)^k = 1^k = 1. \]
Because its ideal class is the identity element, the ideal $I_f^n$ is strictly principal.

\textbf{Case 2 ($n$ is odd):} Let $n = 2k+1$ for some integer $k \ge 0$. Applying the torsion result again yields:
\[ [I_f^n] = [I_f]^{2k+1} = \left([I_f]^2\right)^k \cdot [I_f] = 1^k \cdot [I_f] = [I_f]. \]
Thus, $I_f^n$ resides in the exact same ideal class as $I_f$. If $I_f$ is not principal, then $[I_f] \neq 1$, which immediately implies $[I_f^n] \neq 1$. Consequently, $I_f^n$ cannot be generated by a single element. 

Finally, because $R$ is a Dedekind domain (by \Cref{lem:dedekind}), every fractional ideal can be generated by at most two elements. Since $I_f^n$ requires strictly more than one generator, it must require exactly two.
\end{proof}

By restricting to the Dedekind domain of the curve, we tame the ambient combinatorial complexity entirely, the behavior becomes strictly periodic with period 2.

\subsection{Geometric Interpretation: The Folding Operation}
We have established the rigorous algebraic stability of these ideals via $2$-torsion (\Cref{thm:torsion}). However, what is the geometric source of this obstruction? If we imagine physically \textbf{folding} the curve $C$ along the diagonal line $y = x$, this folding action classifies all points into two distinct types:

\vspace{0.5em}
\noindent\textbf{Off-Diagonal Points (The Split Case).}
Consider a point $P$ off the diagonal. Its mirror image $\sigma P$ lies on the opposite side of the fold. In the Class Group, $[\mathfrak{p}_P]$ represents a \textbf{charge} at $P$, while $[\sigma \mathfrak{p}_P]$ represents the opposite charge. For the ideal $I_f$, these charges combine as $[\mathfrak{p}_P] + [\sigma \mathfrak{p}_P]$, which annihilate each other via the Norm Principle. Off-diagonal geometry is \textit{self correcting}.

\vspace{0.5em}
\noindent\textbf{Diagonal Points (The Ramified Case).}
Consider a point $Q$ exactly on the diagonal $y = x$. When folded, $Q$ does not move; it has no twin to cancel its charge. The diagonal is the only place where \textit{residue} can accumulate. If $f$ cuts through the diagonal an odd number of times, unpaired charges remain, acting as the fundamental obstruction to principality.

\subsection{The Symmetric Index and Parity Criterion}
The abstract algebraic question \textit{Is $I_f$ principal?} collapses into a simple geometric counting problem: \textit{Does the generator $f$ hit the diagonal an even or odd number of times?} 

\begin{defn}[Symmetric Index]
Let $C$ be a symmetric curve. Let $f \in R$ be non-zero. The \textbf{Symmetric Index} of $f$, denoted $\operatorname{Ind}_{sym}(f)$, is the element of $\Cl(R)$ defined by the sum of the classes of the intersection points strictly along the diagonal:
\[ \operatorname{Ind}_{sym}(f) := \sum_{P \in Z(f) \cap \{y=x\}} v_P(f) \cdot [P], \]
where $v_P(f)$ denotes the intersection multiplicity.
\end{defn}

\begin{defn}[Split and Diagonal Divisors]
Let $D$ be the divisor associated to the ideal $I_f$. Based on the classification of prime orbits, $D$ naturally decomposes into $D = D_{split} + D_{diag}$ where,
\begin{enumerate}

    \item $D_{split} = \sum_{\mathcal{O} = \{P, \sigma P\}} v_{\mathcal{O}}(P + \sigma P)$ is the split divisor formed by off-diagonal conjugate pairs.
    \item $D_{diag} = \sum_{P \in \{y=x\}} v_P P$ is the ramified divisor formed strictly by points on the diagonal.
\end{enumerate}
\end{defn}

\begin{thm}[The General Symmetric Index Theorem]
Let $I_f = (f, \sigma f)$ be a Principal Symmetric Ideal. The ideal $I_f$ is a principal ideal if and only if its Symmetric Index vanishes: $\operatorname{Ind}_{sym}(f) = 0 \in \Cl(R).$
\end{thm}
\begin{proof}
Recall from \Cref{thm:factorization} that the divisor of the ideal decomposes into a split divisor $D_{split}$ and a ramified divisor $D_{diag}$. Because $[P] + [\sigma P] = 0$ via the Norm Principle, $[D_{split}] = 0$. The class is entirely determined by the diagonal component: $[I_f] = [D_{diag}] = \operatorname{Ind}_{sym}(f)$. Thus, $I_f$ is principal if and only if $\operatorname{Ind}_{sym}(f) = 0$.
\end{proof}

\begin{cor}[The Parity Criterion]\label{cor:parity}
Let $R$ be the coordinate ring of a hyperelliptic curve. A Principal Symmetric Ideal $I_f$ is principal if and only if the generator $f$ intersects the diagonal with even total multiplicity:
\[ I_f \text{ is principal } \text{if and only if} \sum_{P \in \{y=x\}} v_P(f) \equiv 0 \pmod 2. \]
\end{cor}
\begin{proof}
On a hyperelliptic curve, points on the diagonal correspond to fixed points of the hyperelliptic involution (Weierstrass points). The difference of any two Weierstrass points is the divisor of a rational function, meaning $[P] = [Q]$ in $\Cl(R)$. Let $\mathcal{W}$ be the class of a single Weierstrass point. The Symmetric Index becomes $\left( \sum v_P(f) \right) \cdot \mathcal{W}$. Since $\mathcal{W}$ generates the $2$-torsion subgroup, the term vanishes if and only if the coefficient is even.
\end{proof}

\section*{Acknowledgments}
The author wishes to express gratitude to Prof. Alexandra Seceleanu for reviewing this draft and providing valuable feedback.


\end{document}